\documentclass{amsart}
\usepackage[arrow, matrix, curve]{xy}
\usepackage{array,amsmath,amssymb,amsthm,verbatim,epsfig}
\usepackage{amsfonts}%
\usepackage{graphicx}
\newtheorem{theorem}{Theorem}
\theoremstyle{plain}

\newtheorem{lemma}[theorem]{Lemma}

\newtheorem{proposition}[theorem]{Proposition}

\numberwithin{equation}{section}

\DeclareMathOperator{\Spec}{Spec}

\DeclareMathOperator{\Gl}{Gl}

\DeclareMathOperator{\Mat}{Mat}
\DeclareMathOperator{\Loop}{L}

\renewcommand{\O}{\mathcal{O}}

\newcommand{\plim}{\underleftarrow{\lim}}

\begin{document}

\title[Vector bundles on curves]{A short note on vector bundles on curves}
\author{Martin Kreidl}

\begin{abstract}
In \cite{beauville-laszlo} Beauville and Laszlo give an interpretation of the affine Grassmannian for $\Gl_{n}$ over a field $k$ as a moduli space of, loosely speaking, vector bundles over a projective curve together with a trivialization over the complement of a fixed closed point. In order to establish this correspondence, they use an abstract descent lemma, which they prove in \cite{bl-descente}. It turns out, however, that one can avoid this descent lemma by using a simple approximation-argument, which leads to a more direct prove of the above mentioned correspondence.
\end{abstract}

\maketitle

\section{Introduction}

There is a well-known correspondence between points of the affine Grassmannian for $\Gl_{n}$ and vector bundles on a projective curve together with certain trivializations. Let us recall this correspondence, as Beauville and Laszlo describe it in \cite{beauville-laszlo}.

Let $X$ be a smooth projective curve over $k$, $p\in X$ be a closed point, and choose a uniformizer $z\in \O_{X,p}$. We fix these data for the rest of these notes. For every $k$-algebra $R$ we set
\begin{equation}
\begin{split}
&X_{R} := X\otimes_{\Spec k}\Spec R, \quad X_{R}^{*} := \Spec (\O_{X}(X-\lbrace p\rbrace)\otimes_{k}R),\\
&D_{R} := \Spec R[[z]],\quad D_{R}^{*} := R((z)).
\end{split}
\end{equation}
These data determine a cartesian diagram of schemes

\hspace{\fill}
\begin{equation}\label{diagFormal}
\begin{xy}
\xymatrix{
D_{R}^{*} \ar^{\psi}[r]\ar_{i}[d] & X_{R}^{*} \ar^{j}[d] \\
D_{R} \ar^{f}[r] & X_{R}.
}
\end{xy}
\end{equation}
\hspace{\fill}

Beauville and Laszlo prove the following

\begin{proposition}[\cite{beauville-laszlo}, Proposition 1.4]\label{thmBL1}
The functor
$$
\Loop\Gl_{n}: R \mapsto \Gl_{n}(R((z)))
$$
on the category of $k$-algebras is isomorphic to the functor which associates to $R$ the set of isomorphism classes of triples $(E,\rho,\sigma)$, where $E$ is a vector bundle of rank $n$ over $X_{R}$, and $\rho$ and $\sigma$ are trivializations of $E$ over $X_{R}^{*}$ and $D_{R}$, respectively.
\end{proposition}

As a consequence they obtain

\begin{proposition}[\cite{beauville-laszlo}, Proposition 2.1 and Remark 2.2]\label{thmBL2}
The affine Grassmannian for $\Gl_{n}$, which is by definition the fpqc-sheafification of the functor $R\mapsto \Gl_{n}(R((z)))/\Gl_{n}(R[[z]])$, is isomorphic to the functor which associates to $R$ the set of isomorphism classes of pairs $(E,\rho)$, where $E$ is a vector bundle of rank $n$ over $X_{R}$, and $\rho$ is a trivialization of $E$ over $X_{R}^{*}$.
\end{proposition}

The interesting part in the proof of Proposition \ref{thmBL1} is to see why the data of trivial vector bundles of rank $n$ on $D_{R}$ and $X_{R}^{*}$, respectively, together with a transition function over $X_{R}^{*}$, determine a vector bundle on $X_{R}$. This is not a classical descent situation, since if $R$ is not Noetherian, $D_{R}$ is in general not flat over $X_{R}$. In \cite{bl-descente} Beauville and Laszlo prove that descent holds nontheless.

In the present notes we present an alternative proof of Proposition \ref{thmBL1} using the following strategy. We define the subring $A_{R}\subset R[[z]]$ as a certain localization of $\O_{X,p}\otimes_{k}R$, which depends functorially on $R$ and determines a flat neighborhood of the locus $z=0$ in $X_{R}$. Let us write $\Delta_{R} = \Spec A_{R}$ and $\Delta_{R}^{*}=\Spec A_{R}[1/z]$.
Then $\Delta_{R}\coprod X_{R}^{*}\to X_{R}$ is an fppf-covering, and if we could replace $D_{R}$ by $\Delta_{R}$ and $D_{R}^{*}$ by $\Delta_{R}^{*}$ in the formulation of Proposition \ref{thmBL1}, then this proposition would immediately follow by faithfully flat descent.
Indeed, we will show below how to arrive at this situation using a simple approximation argument. Moreover, the concrete situation will turn out to be not only fppf-local, but even Zariski-local, so that descent of vector bundles holds trivially.

\section{Vector bundles on a smooth curve}\label{sectionAlgebraic}

Note that the choice of a uniformizer $z\in \O_{X,p}$ determines an inclusion $(R\otimes_{k}\O_{X,p}) \subset R[[z]]$, $R[[z]]$ being the completion with respect to the $z$-adic valuation.
For each $f\in (R\otimes_{k}\O_{X,p})\cap R[[z]]^{\times}$
we define $S_{R,f} := (R\otimes_{k}\O_{X,p})_{f} \subset R[[z]]$. The union of all these rings, for varying $f$, will be denoted $A_{R}$. Writing $\Delta_{R} := \Spec A_{R}$ and $\Delta_{R}^{*} := \Spec A_{R}[1/z]$ we have a cartesian diagram

\hspace{\fill}
\begin{xy}
\xymatrix{
\Delta_{R}^{*} \ar^{\psi}[r]\ar^{\iota}[d] & X_{R}^{*} \ar^{j}[d] \\
\Delta_{R} \ar^{\varphi}[r] &  X_{R}.
}
\end{xy}
\hspace{\fill}

Moreover we set $U_{R,f} := \Spec S_{R,f}$.

\begin{lemma}\label{lemCoverings}
The morphism $D_{R} \coprod X_{R}^{*}\to X_{R}$
is surjective. Thus $\Delta_{R} \coprod X_{R}^{*} \to X_{R}$ is an fppf-, and $U_{R,f} \coprod X_{R}^{*} \to X_{R}$ is a Zariski-covering for each $f\in (R\otimes_{k}\O_{X,p})\cap R[[z]]^{\times}$.
\end{lemma}

\begin{proof}
Let $P$ be a point of $X_{R}$ and let $A=(\O_{X}\otimes R)_{P}$ be the local ring at $P$. Either $z$ is invertible in $A$ -- then $P\in X_{R}^{*}$ -- or $z$ is in the maximal ideal $\mathfrak{p} \subset A$. In the latter case we consider $can: A\to\hat{A}=\plim A/z^{N}$ and the ideal $\hat{\mathfrak{p}} = \plim \mathfrak{p}/z^{N}$. Passing to the inverse limit over the short exact sequences
$$
0 \to \mathfrak{p}/(z^{N}) \to A/(z^{N}) \to A/\mathfrak{p} \to 0
$$
we obtain $can^{-1}(\hat{\mathfrak{p}}) = \mathfrak{p}$, and the commutative square

\hspace{\fill}
\begin{xy}
\xymatrix{
\Spec \hat{A} \ar[r]\ar[d] & \Spec R[[z]]=D_{R} \ar[d] \\
\Spec A \ar[r] &  X_{R}.
}
\end{xy}
\hspace{\fill}

shows that $\hat{\mathfrak{p}}\cap R[[z]] \subset R[[z]]$ is a preimage of $P$ in $D_{R}$.
\end{proof}

Let $T$ be the functor on the category of $k$-algebras, which associates to a $k$-algebra $R$ the set of isomorphisms classes of triples $(E,\rho,\sigma)$, where $E$ is a vector bundle of rank $n$ on $X_{R}$, and
\begin{align*}
\rho: \mathcal{O}_{X_{R}^{*}}^{n} \xrightarrow{\simeq} E_{\rvert X^{*}_{R}},\\
\sigma: \mathcal{O}_{\Delta_{R}}^{n} \xrightarrow{\simeq} E_{\rvert \Delta_{R}}
\end{align*}
are trivializations. To each isomorphism class $[(E,\rho,\sigma)]\in T(R)$ we may assign the respective `transition matrix over $\Delta_{R}^{*}$'. This is independent of the actual representative of $[(E,\rho,\sigma)]$ and hence determines a morphism of functors
$$
\Phi(R): T(R) \to \Gl_{n}(A_{R}[1/z]); \quad (E,\rho,\sigma) \mapsto \Gamma(X_{R}, (\rho\rvert_{\Delta_{R}^{*}})\circ(\sigma^{-1}\rvert_{\Delta_{R}^{*}})).
$$

\begin{proposition}\label{propAlgebraic}
The morphism $\Phi(R)$ defined above is an isomorphism of functors.
\end{proposition}

\begin{proof}
We have to construct an inverse for $\Phi(R)$. To this end, we choose a matrix $g \in \Gl_{n}(A_{R}[1/z])$ and consider the following diagram of quasi-coherent sheaves on $X_{R}$,

\hspace{\fill}
\begin{xy}
\xymatrix{
E \ar[rr]\ar[d] & & \mathcal{O}_{X_{R}^{*}}^{n} \ar^{can}[d] \\
\mathcal{O}_{\Delta_{R}}^{n} \ar^{can}[r] & \mathcal{O}_{\Delta_{R}^{*}}^{n} \ar^{g}[r] & \mathcal{O}_{\Delta_{R}^{*}}^{n},
}
\end{xy}
\hspace{\fill}

where $E$ is uniquely determined up to isomorphism by requiring that the diagram be cartesian. (By abuse of notation we do not indicate the obvious push-forwards to $X_{R}$ in this diagram.) It is easy to check (by pullback to $\Delta_{R}$ and $X_{R}^{*}$, respectively) that this diagram determines trivializations of $E$ over $\Delta_{R}$ and $X_{R}^{*}$. The transition function for these two trivializations is equal to $g$ by construction.

To see that this construction indeed gives an inverse for $\Phi(R)$ it remains to check that $E$ is a vector bundle. This is immediate by Lemma \ref{lemCoverings} together with faithfully flat descent, or by the following elementary argument: the matrix $g$ involves only finitely many elements of $A_{R}[1/z]$, whence in fact $g\in S_{R,f}[1/z]$ for some $f\in (R\otimes_{k}\O_{X,p})\cap R[[z]]^{\times}$. This shows that $E$ can as well be obtained by gluing trivial bundles over $U_{R,f}$ and over $X_{R}^{*}$, respectively. Now, since $U_{R,f} \subset X_{R}$ is Zariski-open, this shows that $E$ is a vector bundle.
\end{proof}

\section{`Formal' descent of vector bundles}\label{sectionFormal}

Let us now consider the situation introduced at the beginning in diagram \eqref{diagFormal}, where we consider the formal neighborhood $D_{R} = \Spec R[[z]]$ of $\Spec R\times\lbrace p\rbrace \subset X_{R}$.

By $\hat{T}$ we denote the functor, which associates to every $k$-algebra $R$ the set of isomorphism classes of triples $(E,\rho,\sigma)$, where $E$ is a vector bundle of rank $n$ over $X_{R}$ and 
\begin{align*}
\rho: \mathcal{O}_{X_{R}^{*}}^{n} \xrightarrow{\simeq} E_{\rvert X_{R}^{*}},\\
\sigma: \mathcal{O}_{D_{R}}^{n} \xrightarrow{\simeq} E_{\rvert D_{R}}
\end{align*}
are trivializations.

As in the previous section, we obtain a functorial morphism $\hat{\Phi}(R): \hat{T}(R) \to \Gl_{n}(R((z)))$ by assigning to each triple $(E,\rho,\sigma)$ the corresponding transition function over $D_{R}^{*}$.

\begin{theorem}[\cite{beauville-laszlo}, Proposition 1.4]\label{thmFormal}
The morphism $\hat{\Phi}$ is an isomorphism of functors.
\end{theorem}

\begin{proof}
In order to construct an inverse for $\hat{\Phi}$, i.e. to construct a triple $(E,\rho,\sigma)$ from a given $\gamma \in \Gl_{n}(R((z)))$, we proceed exactly as in the proof of Proposition \ref{propAlgebraic}. The only non-trivial thing to check is that the quasi-coherent sheaf $E$, defined so to make the diagram

\hspace{\fill}
\begin{equation}\label{diagX}
\begin{xy}
\xymatrix{
E \ar[rr]\ar[d] & & \mathcal{O}_{X_{R}^{*}}^{n} \ar^{can}[d] \\
\mathcal{O}_{D_{R}}^{n} \ar^{can}[r] & \mathcal{O}_{D^{*}_{R}}^{n} \ar^{\gamma}[r] & \mathcal{O}_{D^{*}_{R}}^{n},
}
\end{xy}
\end{equation}
\hspace{\fill}

cartesian, is a vector bundle over $X_{R}$. We do this by reducing to a situation where Proposition \ref{propAlgebraic} applies.
More precisely, Lemma \ref{lemDensity} below shows that every $\gamma \in \Gl_{n}(R((z)))$ can be written as a product $\gamma = g\cdot \delta$, where $g\in \Gl_{n}(A_{R}[1/z])$ and $\delta\in \Gl_{n}(R[[z]])$.

Thus diagram \eqref{diagX} `decomposes' likewise, and yields the big diagram

\hspace{\fill}
\begin{xy}
\xymatrix{
E \ar@{=}[r]\ar[dd] & E \ar[rr]\ar[d] & & \mathcal{O}_{X_{R}^{*}}^{n} \ar^{can}[d] \\
& \mathcal{O}_{\Delta_{R}}^{n} \ar^{can}[r]\ar[d] & \mathcal{O}_{\Delta_{R}^{*}}^{n} \ar^{g}[r]\ar[d] & \mathcal{O}_{\Delta_{R}^{*}}^{n}\ar[d] \\
\mathcal{O}_{D_{R}}^{n} \ar^{\simeq}_{\delta}[r] & \mathcal{O}_{D_{R}}^{n} \ar^{can}[r] & \mathcal{O}_{D^{*}_{R}}^{n} \ar_{g}[r] & \mathcal{O}_{D^{*}_{R}}^{n}.
}
\end{xy}
\hspace{\fill}

The two small squares in this diagram are trivially cartesian, while the big rectangle coincides with the square \eqref{diagX}, and is thus cartesian by definition of $E$. Consequently, the upper rectangle is cartesian, which proves that $E$ is nothing but the vector bundle corresponding to the transition matrix $g\in \Gl_{n}(A_{R}[1/z])$ under the correspondence of Proposition \ref{propAlgebraic}.
\end{proof}

\begin{lemma}\label{lemDensity}
We have $\Gl_{n}(R((z))) = \Gl_{n}(A_{R}[1/z])\cdot \Gl_{n}(R[[z]])$.
\end{lemma}

\begin{proof}
We set $B := \displaystyle\cup_{P\in R[z]\cap R[[z]]^{\times}} R[z,z^{-1},P^{-1}] \subset R((z))$ (Note that the ring $B\cap R[[z]]$ is equal to the ring $A_{R}$ in the case $X=\mathbb{P}^{1}_{k}$.). Since $B\subset A_{R}[1/z]$, it suffices to check that 
$\Gl_{n}(R((z))) = \Gl_{n}(B)\cdot \Gl_{n}(R[[z]])$. First we note that $\Gl_{n}(R[[z]])\subset \Gl_{n}(R((z)))$ is open: Namely, $\det: \Mat_{n}(R[[z]]) \to R[[z]]$ is continuous and $R$ carries the discrete topology, and thus $R^{\times} \subset R$ is open. This shows that $\Gl_{n}(R[[z]]) \subset \Mat_{n}(R[[z]]) \subset \Mat_{n}(R((z)))$ are two open inclusions, so $\Gl_{n}(R[[z]])\subset \Gl_{n}(R((z)))$ is as well open. As a second step we deduce from Lemma \ref{lemUnits} below that $\Gl_{n}(B) = \Gl_{n}(R((z)))\cap \Mat_{n}(B)$. Since $\Mat_{n}(B)\subset \Mat_{n}(R((z)))$ is dense and $\Gl_{n}(R((z)))\subset \Mat_{n}(R((z)))$ is open, we conclude that $\Gl_{n}(B) \subset \Gl_{n}(R((z)))$ is dense.

These two statements together imply that $\Gl_{n}(B)\cdot \Gl_{n}(R[[z]])$ is dense and closed in $\Gl_{n}(R((z)))$, whence the lemma.
\end{proof}

\begin{lemma}\label{lemUnits}
The subring $B\subset R((z))$ defined above satisfies
$B^{\times} = R((z))^{\times} \cap B.$
\end{lemma}

\begin{proof}
We consider $f\in R((z))^{\times} \cap B$. By multiplying with a suitable $P\in R[z]\cap R[[z]]^{\times}$, we may reduce to the case $f\in R((z))^{\times} \cap R[z,z^{-1}]$. Such an $f$ has the form $f = -N + Q$, where $N\in R[z,z^{-1}]$ is a nilpotent Laurent polynomial and the leading coefficient of $Q\in R((z))^{\times}$ is a unit in $R$. Using the formula $(-N+Q)(N^{i}+N^{i-1}Q+\dotsb+Q^{i})=(-N^{i}+Q^{i})$ we may assume that $f=Q^{i}$, i.e. has a leading coefficient in $R^{\times}$. Multiplying with $z^{m}$ for a suitable $m\in\mathbb{Z}$ we obtain $z^{m}f \in R[z]\cap R[[z]]^{\times}$, which is invertible in $B$ by construction.
\end{proof}

The property of the ring $B$ which is exhibited in the last lemma is crucial for our strategy of approximation to work. This is what forces us to consider the, at first glance, rather artificial rings $A_{R}$ instead of for example just $\O_{X,p}\otimes R$. The latter would not contain the ring $B$, and in particular would not have the property of Lemma \ref{lemUnits}.

\end{document}